\documentclass[centertags,12pt]{amsart}
\usepackage{latexsym}
\usepackage{amssymb,verbatim}

\numberwithin{equation}{section} \makeatletter
\renewcommand{\subsection}{\@startsection
{subsection}{2}{0mm}{\baselineskip}{-0.25cm}
{\normalfont\normalsize\bf}} \makeatother
\newcommand{\be}{\begin{eqnarray}}
\newcommand{\ee}{\end{eqnarray}}
\newcommand{\ben}{\begin{eqnarray*}}
\newcommand{\een}{\end{eqnarray*}}
\numberwithin{equation}{section}

\newtheorem{theorem}{Theorem}[section]
\newtheorem{proposition}[theorem]{Proposition}
\newtheorem{corollary}[theorem]{Corollary}
\newtheorem{lemma}[theorem]{Lemma}
\theoremstyle{definition}
\newtheorem{remark}[theorem]{Remark}

\def\cD{\mathcal D}

\def\cL{\mathcal L}

\def\cX{\mathcal X}

\def\fq{{\mathbb F}_q}

\def\deg{{\rm deg}}

\def\fq{{\mathbb F}_q}
\def\fqq{{\mathbb F}_{q^2}}

\begin{document}

\title[]{On some open problems on maximal curves}

\author{Stefania~Fanali~and~Massimo~Giulietti}
\thanks{S. Fanali and M. Giulietti are with the Dipartimento di Matematica e Informatica, Universit\`a di  Perugia,
Via Vanvitelli 1, 06123, Perugia,
Italy (e-mail: stefania.fanali@dipmat.unipg.it;  giuliet@dipmat.unipg.it)}
\thanks{This research was performed within the activity of GNSAGA of the
Italian INDAM.}
\maketitle

        \begin{abstract} In this paper we solve three open problems on maximal curves with Frobenius dimension $3$.  In particular, we prove the existence of a maximal curve with order sequence $(0,1,3,q)$.
        \end{abstract}

\maketitle

\section{Introduction}

Let $\fqq$ be a finite field  with $q^2$ elements where $q$ is a
power of a prime $p$. An $\fqq$-rational curve, that is a
projective, geometrically
absolutely irreducible, non-singular algebraic curve defined over
$\fqq$, is called $\fqq$-maximal if the number of its
$\fqq$-rational points attains the  Hasse-Weil upper bound
$$q^2+1+2gq$$
where $g$ is the genus of the curve. Maximal curves have
interesting properties and have also been investigated for their
applications in Coding theory. Surveys on maximal curves are found
in
\cite{garcia2001sur,garcia2002,garcia-stichtenoth1995ieee,geer2001sur,geer2001nato}
and \cite[Chapter 10]{HKT}; see also
\cite{FGT,ft,garcia-stichtenoth2007book,r-sti,sti-x}.

For an $\fqq$-maximal curve $\cX$, the {\em Frobenius linear series}
is the complete linear series $\cD=|(q+1)P_0|$, where $P_0$ is any
$\fqq$-rational point of $\cX$. The projective dimension $r$ of the
Frobenius linear series, called the {\em Frobenius dimension} of
$\cX$, is one of the most important birational invariants of maximal
curves. No maximal curve with Frobenius dimension $1$ exists, whereas the
Hermitian curve is the only maximal curve with Frobenius dimension
$2$. Maximal curves with higher Frobenius dimension have small genus, see Proposition \ref{castel}.

In this paper, we deal with some open problems concerning maximal curves $\cX$ with Frobenius
dimension $3$. For $P \in \cX$ denote by $j_i(P)$ the $i$-th $(\cD,P)$-order and by $\epsilon_i$ the $i$-th $\cD$-order ($i=0,\ldots, 3$). For $i\neq 2$, the values of $\epsilon_i$ and $j_i(P)$ are known, see e.g. \cite[Prop. 10.6]{HKT}. More precisely, $\epsilon_0=0$, $\epsilon_1=1$ and $\epsilon_3=q$; for an $\fqq$-rational point $P\in \cX$, $j_0(P)=0$, $j_1(P)=1$, $j_3(P)=q+1$; for a non-$\fqq$-rational point $P$,  $j_0(P)=0$, $j_1(P)=1$, $j_3(P)=q$.

 In 1999, Cossidente, Korchm\'aros and Torres \cite{CKT1} proved that $\epsilon_2$ is either $2$ or $3$, and that if the latter case holds then $p=3$. They also showed that for an $\fqq$-rational point $P\in \cX$ only a few possibilities for $j_2(P)$ can occur, namely
$$
j_2(P)\in \left\{2,3,q+1-\left\lfloor \frac{1}{2}(q+1)\right\rfloor,q+1-\left\lfloor \frac{2}{3}(q+1)\right\rfloor\right\}.
$$
In \cite{CKT1} it was asked whether the following three cases actually occur for maximal curves with Frobenius dimension $3$:
\begin{itemize}
\item[(A)] $\epsilon_2=3$;

\item[(B)] $\epsilon_2=2$, $j_2(P)=3$ for some  $\fqq$-rational point $P$; 

\item[(C)] $\epsilon_2=2$, $j_2(P)=q+1-\left\lfloor \frac{2}{3}(q+1)\right\rfloor$ for some  $\fqq$-rational point $P$.

\end{itemize}

The main result of the paper is the proof that the recently discovered GK-curve \cite{GK} defined over ${\mathbb F}_{27^2}$ provides an affirmative answer to question (A), see Theorem \ref{3}. It is also shown that the curve of equation $
Y^{16}=X(X+1)^6$ defined over ${\mathbb F}_{49}$ provides an affirmative answer to both questions (B) and (C), see Theorem \ref{7}.
Finally, in Section \ref{finale} we construct an infinite family of maximal curves with $\cD$-orders $(0,1,2,q)$  having an $\fqq$-rational point $P$ with 
$j_2(P)=3$, see Theorem \ref{nuovo1}.

It should be noted that in \cite[Section 4]{AT} it is pointed out that due to some results by Homma and Hefez-Kakuta, an interesting geometrical property of a maximal curves $\cX$ with Frobenius dimension $3$ with 
$\epsilon_2=3$ is that of being a non-reflexive space curve of degree $q+1$ whose tangent surface is
also non-reflexive. 

The language of function fields will be used throughout the paper. The points of a maximal curve $\cX$ will be then identified with the places of the function field $\fqq(\cX)$. Places of degree one correspond to $\fqq$-rational points.

\section{Preliminaries}
Throughout the paper, $p$ is a prime number, $q=p^n$ is some power of $p$, $K={\mathbb F}_{q^2}$ is the finite
field with $q^2$ elements, $F$ is a function field over $K$ such that  $K$ is algebraically closed in $F$,
$g(F)$ is the genus of $F$, $N(F)$ is the number of places of degree $1$ of $F$, ${\mathbb P}(F)$ is the set
of all places of $F$.

For a place $P$ of degree $1$, let $H(P)$ be the Weierstrass semigroup at $P$, that is, the set of non-negative integers $i$ for which there exists $\alpha \in F$ such that the pole divisor $(\alpha)_\infty$ is equal to $iP$.

For a divisor $D$ of $F$, let $\cL(D)$ be the Riemann-Roch space of $D$, see e.g. \cite[Def. 1.4.4]{STI}. The set of effective divisors
$
|D|=\{\alpha+D\mid \alpha \in \cL(D)\}
$
is the complete linear series associated to $D$. The degree $n$ of $|D|$ is the degree of $D$, whereas the dimension $r$ of $|D|$ is the dimension of the $K$-linear space $L(D)$ minus $1$.

We recall some facts on orders of linear series, for which we refer to \cite[Section 7.6]{HKT}.
For a place $P$ of $F$, an integer $j$ is a $(|D|,P)$-order if there exists a divisor $E$ in $|D|$ with $v_P(E)=j$. There are exactly $r+1$ orders
$$
j_0(P)<j_1(P)<\ldots<j_r(P),
$$
and $(j_0(P),j_1(P),\ldots,j_r(P))$ is said to be the $(|D|,P)$-order sequence. For all but a finite number of places the $(|D|,P)$-order sequence is the same. Let $(\epsilon_0,\ldots,\epsilon_r)$ be the generic $(|D|,P)$-order sequence, called the $|D|$-order sequence. In general, $j_i(P)\ge \epsilon_i$. The so-called {\em $p$-adic criterion} (see e.g. \cite[Lemma 7.62]{HKT}) states that if $\epsilon<p$ is a $|D|-$order, then $0,1,\ldots,\epsilon-1$ are also $|D|$-orders.

Let $F$ be a maximal function field, that is, $N(F)=q^2+1+2gq$. 
For a place $P_0$ of degree $1$, let $\cD=|(q+1)P_0|$ be the Frobenius linear series of $F$. By the so-called fundamental equation (see e.g. \cite[Section 9.8]{HKT})
the linear series $\cD$ does not depend on the choice of $P_0$.  The dimension $r$ of $\cD$ is the Frobenius dimension of $F$.
Some facts on the Frobenius linear series of a maximal function field are collected in the following proposition (see \cite[Prop. 10.6]{HKT}).
\begin{proposition}\label{collect} Let $\cD$ be the Frobenius linear series of a maximal function field $F$, and let $(\epsilon_0,\ldots,\epsilon_r)$ be the order sequence of $\cD$. For a place $P$ of degree $1$, let 
$$
H(P)=\{0=m_0(P)<m_1(P)<m_2(P)<\ldots\}.
$$
\begin{itemize}
\item[(a)] $m_r(P)=q+1$, $m_{r-1}(P)=q$.
\item[(b)] The $(\cD,P)$-orders at a place $P$ of degree $1$ are the terms of the sequence
$$
0<1<q+1-m_{r-2}(P)<\ldots<q+1-m_1(P)<q+1.
$$
\item[(c)] $\epsilon_0=0$, $\epsilon_1=1, \epsilon_r=q$.
\end{itemize}

\end{proposition}

The only maximal function field with Frobenius dimension $2$ is the Hermitian function field $H=K(x,y)$ with $y^{q+1}=x^q+x$, see e.g. \cite[Remark 10.23]{HKT}. Maximal function fields with Frobenius dimension $3$ were investigated in \cite{CKT1}. Corollary 3.5 in  \cite{CKT1} states that if $\epsilon_2=2$, then for any place $P$ of degree $1$
$$
j_2(P)\in \left \{2,3,q+1-\left\lfloor \frac{1}{2}(q+1)\right\rfloor, q+1-\left\lfloor \frac{2}{3}(q+1)\right\rfloor\right\}.
$$
For each value of $q$ there exists a unique maximal function field such that $j_2(P)=q+1-\left\lfloor \frac{1}{2}(q+1)\right\rfloor$ holds for some place $P$ (see \cite[Remark 3.6]{CKT1}). A number of examples for which $j_2(P)=2$ occurs are known, see \cite[Chapter 10]{HKT}. So far, no example of a maximal function field with Frobenius dimension $3$ having a place $P$ of degree $1$ with $j_2(P)\in \left \{3, q+1-\left\lfloor \frac{2}{3}(q+1)\right\rfloor\right\}$ appears to have been known in the literature (see \cite[Remark 3.9]{CKT1}, \cite[Section 4]{AT}).

A result from \cite{CKT1} that will be useful in the sequel is the following.
\begin{lemma}\cite[Lemma 3.7]{CKT1}\label{fern} If the Frobenius dimension of a maximal function field is $3$, then there exists a place $P$ of degree $1$ with $j_2(P)=\epsilon_2$.
\end{lemma}
Maximal function fields with higher Frobenius dimension have smaller genus, as stated in the next result.
\begin{proposition}\cite[Corollary 10.25]{HKT}\label{castel}
The genus $g$ of a maximal function field with Frobenius dimension $r$ is such that
$$
g\le \left\{
\begin{array}{ll}
\frac{(2q-(r-1))^2-1}{8(r-1)} & \quad\text{ if r is even,}\\
&  \\
\frac{(2q-(r-1))^2}{8(r-1)} & \quad\text{ if r is odd.}
\end{array}
\right.
$$

\end{proposition}
\section{The $\cD$-order sequence of the GK function field}

Throughout this section, we assume that $q={\bar q}^3$ with ${\bar q}$ a prime power.
Let $F$ be the function field $K(x,y)$, where $y^{{\bar q}+1}=x^{\bar q}+x$. Let $u=y\frac{x^{{\bar q}^2-1}-1}{x^{{\bar q}-1}+1}$, and
consider the field extension
$F(z)/F $ where $z^{{\bar q}^2-{\bar q}+1}=u$. The GK function field is
\begin{equation}
{\bar F}=F(z).
\end{equation}


We first recall some proprieties of ${\bar F}$, for which we refer to \cite[Section~2]{GK}.
The function field ${\bar F}$ is a Kummer extension of $F$, and in particular ${\bar F}/F$ is Galois of degree ${\bar q}^2-{\bar q}+1$.
The Galois group $\Gamma$ of ${\bar F}/F$ consists of all the automorphisms $g_u$ of ${\bar F}$ such that
$$
g_u(x)=x, \qquad g_u(y)=y, \qquad g_u(z)=uz,
$$
with $u^{\bar{q}^{2}-\bar{q}+1}=1$.

The function field ${\bar F}$ is $\fqq$-maximal.
Significantly, for $q>8$, ${\bar F}$ is the only known function field that is maximal but not a subfield of the Hermitian function field (see \cite[Theorem~5]{GK}). The genus of ${\bar F}$ is $$g=\frac{1}{2}(\bar{q}^{3}+1)(\bar{q}^{2}-2)+1.$$
Also, the only common pole of $x,y$ and $z$ is a place $P_0$ of degree $1$ for which  $$\cL((q+1)P_0)=<1,x,y,z>.$$ Therefore the Frobenius linear series $\cD$ consists of divisors
$$
\cD=\{div(a_0+a_1x+a_2y+a_3z)+(q+1)P_0\mid (a_0,a_1,a_2,a_3)\in K^4, (a_0,a_1,a_2,a_3)\neq(0,0,0,0)\}.
$$

Let $P'$ be any place of degree $1$ of ${\bar F}$. Let $P$ be the place of $F$ lying under $P'$. Then
\begin{equation}
\left\{ \begin{array}{l}e(P'\mid P)={\bar q}^2-{\bar q}+1, \quad \text{if }P\text{ is either a zero or a pole of  }z\\ e(P'\mid P)=1, \quad
\text{otherwise}\end{array}\right..
\end{equation}

We now describe the $(\cD,P')$-orders for a place $P'$ of degree $1$ of ${\bar F}$.

%
\begin{proposition}\cite[Section 4]{GK}\label{0}
If  $P'$ is such that $e(P'\mid P)={\bar q}^2-{\bar q}+1$, then the Weierstrass semigroup at $P$ is the subgroup generated by $\bar{q}^{3}-\bar{q}^{2}+\bar{q}$, $\bar{q}^{3}$, and $\bar{q}^{3}+1$.
\end{proposition}
>From (b) of Proposition \ref{collect} the following corollary is obtained.
\begin{corollary}
If  $P'$ is such that $e(P'\mid P)={\bar q}^2-{\bar q}+1$, then
$$
(j_0(P'),j_1(P'),j_2(P'),j_3(P'))=(0,1,{\bar q}^2-{\bar q}+1,\bar{q}^{3}+1)
$$
\end{corollary}

Assume now that $e(P'\mid P)=1$.  As this occurs for an infinite number of places $P$, it is possible to choose $P$ in such a way that there exists $ay+by+c\in F$ such that $v_P(ay+by+c)={\bar q}$
(see e.g. \cite[p. 302]{HKT}). Then $v_{P'}(ax+by+c)=v_{P}(ax+by+c)=\bar{q}$ holds.
Then by (b) of Proposition \ref{collect}, $\bar{q}^{3}-\bar{q}+1\in H(P')$. Taking into account that the automorphism group of ${\bar F}$ acts transitively on the set of places of degree $1$ with  $e(P'\mid P)=1$ \cite[Theorem 7]{GK}, the following result is obtained.

\begin{proposition}
If  $P'$ is such that $e(P'\mid P)=1$, then
$$
(j_0(P'),j_1(P'),j_2(P'),j_3(P'))=(0,1,{\bar q},{\bar q}^{3}+1)
$$
\end{proposition}

\begin{theorem} If $q$ is a cube, then there exists a  maximal function field with Frobenius dimension $3$ and with $\cD$-order sequence $(0,1,\sqrt[3]q,q)$.
\end{theorem}
\begin{proof} We prove that the $\cD$-order sequence of the GK function field is $(0,1,{\bar q},{\bar q}^3)$. By Lemma \ref{fern}, there exists an ${\mathbb F}_{q^2}$-rational place $P$ of ${\bar F}$ such that $j_2(P)=\epsilon_2$. Since the only possibilities for $j_2(P)$ are ${\bar q}$ and ${\bar q}^2-{\bar q}+1$, and since $j_2(P)\ge \epsilon_2$ for every $P\in {\mathbb P}({\bar F})$, the claim follows.
\end{proof}
Therefore, the answer to question (A) in Introduction is obtained.
\begin{theorem}\label{3} There exists a maximal curves over ${\mathbb F}_{27^2}$ with Frobenius dimension $3$ and with $\cD$-order sequence $(0,1,3,27)$.
\end{theorem}

\section{On a maximal function field over ${\mathbb F}_{49}$}
In \cite[Example 6.3]{GSX} it is shown that for every divisor $m$ of $q^2-1$ the function field $F=K(x,y)$ with 
$$
y^{\frac{q^2-1}{m}}=x(x+1)^{q-1}
$$
is a maximal function field with  genus $g=\frac{1}{2m}(q+1-d)(q-1)$, where $d={\rm gcd}(m,q+1)$. In this section we focus on the case $q=7$ and  $m=3$, whence $d=1$ and $g=7$. We are going to prove the following result, which provides an affirmative answer to both questions (B) and (C) in Introduction.
\begin{theorem}\label{7} Let $F=K(z,t)$ be the function field  defined over $K={\mathbb F}_{49}$ by the equation $z^{16}=t(t+1)^6$. Then the Frobenius dimension of $F$ is $3$, the $\cD$-order sequence of $F$ is $(0,1,2,7)$, and there exists an ${\mathbb F}_{49}$-rational place $P$ of $F$ such that $j_2(P)=3=8-\lfloor\frac{2}{3}(8)\rfloor$.
\end{theorem}
The function field $F$ is a subfield of the Hermitian function field $H=K(x,y)$ with $y^{8}=x^7+x$. More precisely,  $F\cong K(x^6,y^3)$, and $H/F$ is Galois of degree $3$ (cf. \cite[Example 6.3]{GSX}). The Galois group of $H/F$ is $\Gamma=\{1,\tau,\tau^2\}$, where 
$\tau(x)=a^8x$, $\tau(y)=ay$, with $a$ a primitive cubic root of unity.

Let $P_0$ (resp. $P_\infty$) be the only zero (resp. pole) of $x$ in $H$. Let $P_1,\ldots,P_{6}$ be the zeros of $y$ in $H$ distinct from $P_0$.
\begin{lemma}  The only ramification points of  $H/F$ are $P_0$ and $P_\infty$. 
\end{lemma}
\begin{proof}
It is easy to see that for each point $P$ of $H$ distinct from $P_0$ and $P_\infty$ the stabilizer of $P$ in $\Gamma$ is trivial. On the other hand, both $P_0$ and $P_\infty$ are fixed by $\Gamma$.
\end{proof}

Let ${\bar P}_0$ and ${\bar P}_\infty$ be the places of $F$ lying under $P_0$ and $P_\infty$, respectively. Let ${\bar P}_1$ and ${\bar P}_2$ be the two places of $F$ lying under the places $P_i$ of $H$, $i=1,\ldots,6$. Also, let $z=y^3$ and $t=x^6$ in $F$. Then
\begin{eqnarray*}
v_{{\bar P}_0}(z)=\frac{1}{3} v_{P_0}({ y}^3)=1,\quad
v_{{\bar P}_0}(t+1)=\frac{1}{3} v_{P_0}({ x}^{6}+1)=0;
\\v_{{\bar P}_\infty}(z)=\frac{1}{3} v_{P_\infty}({ y}^3)=-7,\quad
v_{{\bar P}_\infty}(t+1)=\frac{1}{3} v_{P_\infty}({ x}^6+1)=\frac{6}{3}ord_{P_\infty}({x}) =-16;\\
{\text{ for }}\,\, i=1,2,\,\,\,
v_{{\bar P}_i}(z)=v_{P_i}({ y}^3)=3,\quad
v_{{\bar P}_i}(t+1)= v_{P_i}({ x}^6+1)=7.
\end{eqnarray*}
To sum up,
$$
(z)=3({\bar P}_1+{\bar P}_2)+{\bar P_0}-7{\bar P}_{\infty}, \qquad (t+1)=8({\bar P}_1+{\bar P}_2)-16{\bar P}_{\infty} .
$$

\begin{proposition}\label{totti} Let $i,j$ be non-negative integers such that
$3 i\ge 8j$.  
Then $7i-16j\in H({\bar P}_\infty)$.
\end{proposition}
\begin{proof}
Let $\gamma=z^i(t+1)^{-j}$. Then
$$
(\gamma)=3i({\bar P}_1+{\bar P}_2)+{\bar P_0}-7i{\bar P}_{\infty}-8j({\bar P}_1+{\bar P}_2)+16j{\bar P}_{\infty},
$$
whence
$$
(\gamma)_\infty=(7i-16j){\bar P}_{\infty}.
$$
\end{proof}

\begin{corollary}\label{SMG} The only non-gaps at ${\bar P_\infty}$ that are less than or equal to $8$ are $0,5,7,8$.
\end{corollary}
\begin{proof} The integers $7$ and $8$ are non-gaps since $F$ is an ${\mathbb F}_{49}$-maximal function field (see Proposition \ref{collect}). Proposition \ref{totti} for $i=3$ and $j=1$ implies that $5$ is a non-gap at ${\bar P}_\infty$. Then it is easy to see that $10,12,13$ are non-gaps as well. Therefore, we have $7$ non-gaps less than $2g=14$. Since $g=7$, this rules out the possibility that there is another positive non-gap less than $7$ and distinct from $5$.
\end{proof}

We are now in a position to prove Theorem \ref{7}.

{\em Proof of Theorem \ref{7}} The Frobenius dimension of $F$ is $3$ by Corollary \ref{SMG}. By the $p$-adic criterion (see Section 2), the $\cD$-order sequence is $(0,1,2,7)$. Again by Corollary \ref{SMG} we have $j_2({\bar P}_\infty)=3$.

\section{An $\fqq$-maximal function field of genus $\frac{q^{2}-q+4}{6}$}\label{finale}

Througouth this section we assume that that $q\equiv 2 \pmod 3$.
We recall some facts about the function field $F$ over
$K$, defined by
$$
F=K(x,y)\qquad \text{with } y^\frac{q+1}{3}+x^\frac{q+1}{3}+1=0.
$$
Clearly $F$ is a subfield of the hermitian function field $H$ over $K$ defined by
$
H=K(z,t) $ with $ z^{q+1}+t^{q+1}+1=0,
$
and therefore $F$ is a maximal function field. Since the equation $Y^\frac{q+1}{3}+X^\frac{q+1}{3}+1=0$ defines a non-singular plane algebraic curve of degree $\frac{q+1}{3}$, the genus $g(F)=\frac{1}{2}\left(\frac{q+1}{3}-1\right)\left(\frac{q+1}{3}-2\right)$, and therefore $N(F)=q^2+1+q\left(\frac{q+1}{3}-1\right)\left(\frac{q+1}{3}-2\right)$. 

It is straightforward to check that the zeros of $x$ are $\frac{q+1}{3}$ distinct places of degree $1$. The same holds for $y$. The pole set of $x$ coincides with the pole set of $y$, and consists of $\frac{q+1}{3}$ places of degree $1$.

For $\alpha, \beta\in {\mathbb F}_{q^2}$ such that
$\alpha^\frac{q+1}{3}+\beta^\frac{q+1}{3}+1=0$, let $P_{\alpha,\beta}\in {\mathbb
P}({F})$  denote the common zero of $x-\alpha$ and
$y-\beta$. Let $P_{\infty,1},\ldots,P_{\infty,\frac{q+1}{3}}$ be the poles of $x$ (and $y$).
Clearly, $P_{\alpha,\beta}$ is a place of degree $1$. 
Also, for any $\beta\in \fqq$ with $\beta^{\frac{q+1}{3}}+1=0$,
the zero divisor $(y-\beta)_0$ of $y-\beta$ is equal to $\frac{q+1}{3}P_{0,\beta}$.

Henceforth, $w$ is an element in $\fqq$ such that $w^\frac{q+1}{3}=3$.
 Let
$$
u=wxy\in F.
$$
For any  place $P$ of $F$ which is a zero of either $x$ or $y$,  
$v_P(u)=1$ holds. Moreover, for any common pole $P$ of $x$ and $y$ we have  $v_{P}(u)=-2$. Any other place of $F$
is neither a pole or a zero of $u$.

Consider the field extension
$F(z)/F $ where $z^{3}=u$. Let
\begin{equation}\label{effe}
{\bar F}=F(z).
\end{equation}
Clearly, $u$ is not a $3$-rd power of an element in $F$. Then ${\bar F}$ is a Kummer
extension of $F$ (see \cite[Proposition III.7.3]{STI}), and in particular ${\bar F}/F$ is Galois of degree $3$.
The ramification index $e(P'\mid P)$ can be easily computed for any
place $P'$ of ${\bar F}$ lying over a place $P$ of $F$: as $gcd(2,3)=1$, (b) of \cite[Proposition III.7.3]{STI}
gives
\begin{equation}\label{ramif}
\left\{ \begin{array}{l}e(P'\mid P)=3, \quad \text{if }P\text{ is either a zero or a pole of } 
xy,\\ e(P'\mid P)=1, \quad
\text{otherwise.}\end{array}\right.
\end{equation}

By \cite[Corollary III.7.4]{STI},
\begin{equation}\label{genere1}
g({\bar F})=1+3(g(F)-1)+3\frac{q+1}{3}=\frac{q^2-q+4}{6}.
\end{equation}

Now we compute the number $N$ of  places of degree $1$ of ${\bar F}$.
Any place in ${\mathbb P}({\bar F})$ of
degree $1$ either lies over some $P_{\infty,i}$, or some $P_{\alpha,\beta}$. 
By (\ref{ramif}), any place lying over either $P_{\infty,i}$ or
$P_{\alpha,\beta}$ with $\alpha\beta=0$ is fully
ramified. This gives $q+1$  places of degree $1$ of ${\bar F}$.

Assume now that $\alpha\beta \neq 0$. Let
$$
\varphi_{\alpha,\beta}(T)=T^3-w\alpha\beta\in
{\mathbb F}_{q^2}[T].
$$
As $gcd(3,p)=1$, $\varphi_{\alpha,\beta}(T)$ has $3$ distinct roots
in the algebraic closure of ${\mathbb F}_{q^2}$. Let $\lambda$ be any
of such roots. Then $\lambda\in {\mathbb F}_{q^2}$ if and only if
\begin{equation}\label{condition}
1=\lambda^{q^2-1}=(\lambda^{q+1})^{q-1}=\left((w\alpha\beta)^\frac{q+1}{3}\right)^{q-1}=\left(3(\alpha\beta)^\frac{q+1}{3}\right)^{q-1},
\end{equation}
that is $3(\alpha\beta)^\frac{q+1}{3}\in \fq$.  
Taking into account the classical relation  $$(A+B+C)\mid A^3+B^3+C^3-3ABC,$$ we have that $\alpha^{\frac{q+1}{3}}+\beta^{\frac{q+1}{3}}+1=0$ yields
$$
3(\alpha\beta)^\frac{q+1}{3}= \alpha^{q+1}+\beta^{q+1}+1.
$$
Then (\ref{condition}) follows since $(\alpha^{q+1}+\beta^{q+1}+1)^q=(\alpha^{q+1}+\beta^{q+1}+1)$.

 By \cite[Proposition III.7.3]{STI},
the minimal polynomial of $z$ over $F$ is $\varphi(T)=T^3-wxy$. As
$wxy \in {\mathcal O}_{P_{\alpha,\beta}}$, Kummer's Theorem
\cite[Theorem III.3.7]{STI} applies, and hence $P_{\alpha,\beta}$
has $3$ distinct extensions $P\in {\mathbb P}({\bar F})$ with $\deg(P)=1$.

Since $F$ is maximal, the number of pairs $(\alpha,\beta)$ with $\alpha\beta\neq 0$ and $\alpha^{\frac{q+1}{3}}+\beta^{\frac{q+1}{3}}+1=0$ is
$$
q^2+1+q\left(\frac{q+1}{3}-1\right)\left(\frac{q+1}{3}-2\right)-(q+1)
$$
Therefore, the total number $N$ of  places of degree $1$ of ${\bar F}$ is
$$N=q+1+3\left( q^2-q+q\left(\frac{q+1}{3}-1\right)\left(\frac{q+1}{3}-2\right)\right)$$
By straightforward computation
$$
N=q^2+1+2q\frac{q^2-q+4}{6},
$$
whence the following result is obtained.
\begin{theorem} ${\bar F}$ is an $\fqq$-maximal function field.
\end{theorem}

\begin{proposition} The Frobenius dimension of ${\bar F}$ is equal to $3$.
\end{proposition}
\begin{proof} The assertion follows from Proposition \ref{castel}.
\end{proof}

\begin{remark}\label{curva_kt}
In \rm{\cite{KT}} $\fqq$-maximal function fields with Frobenius dimension $3$ and genus $\frac{q^{2}-q+4}{6}$. We are not able to tell whether they are isomorphic to ${\bar F}$ or not.
\end{remark}

Fix $\beta \in K$ with $\beta^{\frac{q+1}{3}}=1$, and let $P=P_{0,\beta}\in {\mathbb P}(F)$. Let ${\bar P}$ be the place of ${\bar F}$ lying over $P$.
The pole divisor of $\frac{x}{y-\beta}$ in $F$ is
$\frac{q-2}{3}P$.
Whence ${\bar P}$ is the only pole of $\frac{x}{y-\beta}$ in ${\bar F}$, and
$$
v_{\bar P}\left(\frac{x}{y-\beta}\right)=-(q-2).
$$
This means that $j_2({\bar P})=3$.

Taking into account that by the $p$-adic criterion the third $\cD$-order must be equal to $2$, the following result is arrived at.

\begin{theorem}\label{nuovo1} Let $q$ be odd, $q\equiv 2 \pmod 3$. Then
${\bar F}$ is an $\fqq$-maximal function field with Frobenius dimension $3$ such that
$$
(\epsilon_0,\epsilon_1,\epsilon_2,\epsilon_3)=(0,1,2,q),
$$
and having an $\mathbb{F}_{q^{2}}$-rational  point $P$ with
$$j_{0}(P_\infty)=0,\,\,\,j_{1}(P_\infty)=1,\,\,\,j_{2}(P_\infty)=3,\,\,\,j_{3}(P_\infty)=q+1.$$
\end{theorem}

\end{document}